\documentclass[reqno]{amsart}
\usepackage{diagrams}
\usepackage{amssymb}
\usepackage{mathrsfs}
\usepackage{cite}
\usepackage{tikz-cd}
\usepackage{MnSymbol}
\usepackage{a4wide}
 
\usepackage{xcolor}

\DeclareMathOperator\Z{\mathbb Z}
\newcommand{\Om}{\Omega}

\newcommand{\G}{{\mathbb G}}

\newtheorem{theorem}{Theorem}[section]
\newtheorem{lemma}[theorem]{Lemma}
\newtheorem{cor}[theorem]{Corollary}
\newtheorem{prop}[theorem]{Proposition}
\theoremstyle{definition}

\newtheorem{example}[theorem]{Example}

\theoremstyle{remark}
\newtheorem{remark}[theorem]{Remark}

\newcommand{\dontprint}[1]\relax

\newcommand{\La}{\Lambda}
\newcommand{\Ga}{\Gamma}
\renewcommand{\P}{{\mathbb P}}
\newcommand{\A}{{\mathbb A}}
\newcommand{\wt}{\widetilde}
\newcommand{\ot}{\otimes}

\newcommand{\Hom}{\operatorname{Hom}}

\newcommand{\Aut}{\operatorname{Aut}}
\newcommand{\Ext}{\operatorname{Ext}}

\newcommand{\DD}{{\mathcal D}}

\newcommand{\FF}{{\mathcal F}}
\newcommand{\GG}{{\mathcal G}}
\newcommand{\II}{{\mathcal I}}

\newcommand{\NN}{{\mathcal N}}

\newcommand{\OO}{{\mathcal O}}
\newcommand{\UU}{{\mathcal U}}

\renewcommand{\SS}{{\mathcal S}}

\newcommand{\si}{\sigma}
\newcommand{\de}{\delta}
\newcommand{\sub}{\subset}

\newcommand{\Spec}{\operatorname{Spec}}

\newcommand{\Ber}{\operatorname{Ber}}

\newcommand{\lan}{\langle}
\newcommand{\ran}{\rangle}

\newcommand{\om}{\omega}
\newcommand{\la}{\lambda}
\renewcommand{\a}{\alpha}

\newcommand{\id}{\operatorname{id}}

\renewcommand{\th}{\theta}

\newcommand{\hra}{\hookrightarrow}

\newcommand{\eps}{\epsilon}

\newcommand{\Pic}{\operatorname{Pic}}
\newcommand{\tot}{\operatorname{tot}}
\newcommand{\GL}{\operatorname{GL}}

\numberwithin{equation}{section}
\title{$\A^{0|1}$-torsors, quotients by free $\A^{0|1}$-actions, and embeddings into $\Pi$-projective spaces and super-Grassmannians $G(1|1,n|n)$}
\author{Alexander Polishchuk}
\thanks{Partially supported by the NSF grant DMS-2001224, 
and within the framework of the HSE University Basic Research Program and by the Russian Academic Excellence Project `5-100'.}

\address{
    Department of Mathematics, 
    University of Oregon, 
    Eugene, OR 97403, USA; National Research University Higher School of Economics; and Korea Institute for 
    Advanced Study 
  }
  \email{apolish@uoregon.edu}

\begin{document}

\begin{abstract}
We study embeddability of superschemes into $\Pi$-projective spaces and into supergrassmannians $G(1|1,n|n)$. We give some criteria based on the relation with
$\A^{0|1}$-torsors and $\A^{0|1}$-fibrations. We also prove the existence of nice quotients for free actions of $\A^{0|1}$ on superschemes.
\end{abstract}

\maketitle

\section{Introduction}

A striking difference between supergeometry and the usual geometry is the absence of (super)projective embeddings for many natural proper superschemes,
e.g., for most supergrassmannians. This is explained by the fact that by considering bosonic quotients of superschemes (i.e., considering only even functions)
one often gets nontrivial nilpotent extensions of the corresponding reduced schemes, which tend to be nonprojective.

In some respects an adequate analog of projective schemes in supergeometry is provided by schemes embeddable into supergrassmannians. 
Recall that the supergrassmannian $G(a|b,m|n)$ parametrizes subspaces of dimension $a|b$ in a supervector space of dimension $m|n$.
Thus, by fixing possible values of $a$ and $b$ we get a hierarchy of spaces in supergeometry: for each $a|b$ we can consider {\it $a|b$-embeddable schemes}, i.e.,
those embeddable into $G(a|b,m|n)$ for some $m|n$. Note that the usual projectivity corresponds to $1|0$-embeddability.
One can also consider embeddability into other homogeneous superspaces. For example, Manin considered in \cite[V.6]{Manin} 
{\it $\Pi$-projective spaces} $\P^n_{\Pi}$ (aka {\it $\Pi$-symmetric superprojective space}) which are homogeneous spaces for the simple supergroups of type $Q$.
More generally, one has the $\Pi$-grassmannian $G\Pi(a|a,n|n)$ which is a closed subscheme in $G(a|a,n|n)$, parametrizing $\Pi$-symmetric subspaces in a superspace 
with a $\Pi$-symmetry (see \cite[V.6]{Manin}). Note that $\P^n_{\Pi}=G\Pi(1|1,n+1|n+1)$.

Even the first few stages of this hierarchy, namely $\Pi$-projectivity and $1|1$-embeddability are still poorly understood.
In the present paper we will contribute to their study by providing some general criteria and some new examples. 
Previous works in this direction mostly construct examples of non-projective, and sometimes non-$\Pi$-projective, superschemes (see e.g., \cite{CNR}, \cite{Noja-emb},\cite{PS}). 


It is well known that projectivity is equivalent to the existence of an (even) line bundle whose restriction to the bosonization is ample (see \cite{LPW}, \cite[Prop.\ A.2]{FKP}).
The natural approach to understanding embeddability into $\Pi$-projective spaces and into $G(1|1,n|n)$ 
is by studying $1|1$-bundles (equipped with a $\Pi$-symmetry if we are interested in morphisms to $\Pi$-projective spaces). 
Our main idea is to study $1|1$-bundles by looking at their projectivizations which are locally trivial $\A^{0|1}$-fibrations (since $\P^{1|1}=\A^{0|1}$). 
Similarly, to a $\Pi$-symmetric $1|1$ bundle one can associate an $\A^{0|1}$-torsor.
Thus, we are led to the study of $\A^{0|1}$-fibrations (resp., $\A^{0|1}$-torsors).
More generally, to a vector bundle of rank $a|b$ one can associate an $\A^{0|ab}$-fibration, however, $\A^{0|n}$-fibrations are much harder to study for $n>1$.

Some of the results of this paper are inspired by the result proved in \cite[Sec.\ 4.9]{BHRP}, stating that a superscheme is
embeddable into some supergrassmannian if and only if there exists an $\A^{0|n}$-fibration whose total space is projective.
Furthermore, any $1|1$-embeddable superscheme admits an $\A^{0|1}$-fibration with projective total space. Using this we will get criteria for a superscheme not to be
$1|1$-embeddable. For example, we prove that $\P^2_\Pi\times \P^2_\Pi$ and $G(1|1,m|m)\times G(1|1,n|n)$ for $m\ge 3$, $n\ge 3$, are 
not $1|1$-embeddable.  

On the other hand, we show that there is a close relation between $\Pi$-projectivity and the existence of an $\A^{0|1}$-torsor with projective total space.
This is due to the existence of a natural map
$$\P^{n|n+1}\to \P^n_\Pi$$
which has a structure of an $\A^{0|1}$-torsor. Namely, let  $(V,p)$ be an $(n+1)|(n+1)$-dimensional space with a $\Pi$-symmetry, then this
map sends a $1|0$-dimensional subspace $L\sub V$ to the $1|1$-dimensional $\Pi$-symmetric subspace $L+p(L)\sub V$.
Thus, given an embedding of $X$ into a $\Pi$-projective space, we get an $\A^{0|1}$-torsor over $X$ which has projective total space. We show that
conversely, under some extra assumptions, the existence of such an $\A^{0|1}$-torsor implies $\Pi$-projectivity (see Theorem \ref{Pi-proj-thm}).
This gives criteria for $\Pi$-projectivity.
For example, we show that the quotient of $\P^{m-1|m}\times \P^{n-1|n}$ by the diagonal $\A^{0|1}$-action is $\Pi$-projective (but not projective for $m\ge 3$, $n\ge 3$),
while $G(1|1,2|2)$ is not $\Pi$-projective.

We also consider the natural class of CY supervarieties of dimension $2|2$ associated with projective surfaces $S$ together with rank 2 vector bundles $V$ such that
$\det V\simeq \om_C$. In the case $S=\P^2$, it was shown in \cite{CNR} that all such supervarieties are embeddable into some supergrassmannian.
We give new examples of embeddability and non-embeddability for this class of supervarieties (see Example \ref{CY22-ex} and Prop.\ \ref{CY22-prop}).




One question that naturally arises in this context is the existence of nice quotients by actions of $\A^{0|n}$ on superschemes.
Even for $\A^{0|n}$-actions on affine superspaces there may not exist a nice quotient (e.g., the corresponding ring of invariants is not
necessarily finitely generated). We prove that for a {\it free} $\A^{0|n}$-action on a superscheme $X$ the quotient always exists as a superscheme (of
finite type if $X$ is of finite type), and the corresponding quotient map $X\to X/\A^{0|n}$ is an $\A^{0|n}$-torsor (see Theorem \ref{free-quot-thm}).
This result is an easy consequence of the affine case considered \cite{Zubkov} (in the case $X$ is smooth this is \cite[Thm.\ 1.8]{MO}). 
It can be viewed as an algebraic counterpart of Shander's theorem on rectifying non-vanishing homological vector fields on supermanifolds (see \cite{Shander}).


\medskip

\noindent
{\it Acknowledgments.} I am grateful to Arkady Vaintrob for useful discussions and to Emile Bouaziz for suggesting interesting examples of embeddable CY $2|2$ supervarieties (see
Proposition \ref{CY22-prop}).

\section{$\A^{0|1}$-torsors and $\A^{0|1}$-fibrations}

\subsection{Free actions of $\A^{0|1}$ and $\A^{0|1}$-torsors}


We denote by $\A^{0|1}$ the odd affine line with the standard group structure. This super group scheme is also often denoted as $\G_a^-$.

Let $S$ be a base superscheme. We work in the category of superschemes over $S$.
By definition, an action of $\A^{0|1}_S$ on $X$ is given by a morphism of superschemes
$$\si:\A^{0|1}_S\times_S X\to X$$
satisfying the usual axioms. Since such a morphism is identity on the underlying topological spaces,
we see that the notion of an $\A^{0|1}$-action is local. It is well known that an $A^{0|1}$-action is determined
by the corresponding homological vector field, which is an odd derivation $v$ on $\OO_X$, trivial on functions pulled back from $S$, such that $v^2=0$.
Namely, for $X=\Spec(A)$, the action map $\si$ is given by $\si^*(a)=a+\th\cdot v(a)$, where $\th$ is the
coordinate on $\A^{0|1}$. Equivalently, we can think of the vector field $v$ as an automorphism of $\A^{0|1}\times X$. Then 
the corresponding action is the composition of this automorphism with the projection to $X$.


\begin{prop}\label{free-action-prop} 
Let $X$ be a superscheme over $S$ equipped with an action of $\A^{0|1}_S$,
and let $v$ be the corresponding homological vector field on $X$.
Then the following conditions are equivalent:
\begin{enumerate}
\item the action of $\A^{0|1}_S$ on $X$ is free;
\item locally there exist functions $a_1,\ldots,a_n$ and $b_1,\ldots,b_n$, where $a_i$ and $b_i$ have opposite parity,
such that $\sum_{i=1}^n a_iv(b_i)=1$;
\item locally there exists an odd function $\th$ such that $v(\th)$ is invertible;
\item locally there exists an odd function $\th$ such that $v(\th)=1$;
\item we have the equality of subsheaves of $\OO_X$, $\ker(v:\OO_X\to \OO_X)=v(\OO_X)$.
\end{enumerate}
\end{prop}

\begin{proof}
(1)$\iff$(2). The morphism
$$\A^{0|1}\times_S X\to X\times_S X: (g,x)\mapsto (gx,x)$$
is given on the level of topological spaces by the diagonal embedding, and locally
corresponds to the ring homomoprhism
$$\kappa:\OO_X\ot_{\OO_S} \OO_X\to \OO_X[\eps]: f_1\ot f_2\mapsto (f_1+\eps v(f_1))\cdot f_2.$$
Thus, the action of $\A^{0|1}$ is free if and only if $\kappa$ is surjective.
Now we notice that $\kappa$ fits into a morphism of exact sequences
\begin{diagram}
0 &\rTo{}& \Om_{X/S} &\rTo{}& \OO_X\ot_{\OO_S}\OO_X &\rTo{}& \OO_X &\rTo{}& 0\\
&&\dTo{v}&&\dTo{\kappa}&&\dTo{\id}\\
0 &\rTo{}& \Pi\OO_X &\rTo{\eps\cdot}& \OO_X[\eps] &\rTo{}& \OO_X &\rTo{}& 0
\end{diagram}
Hence, $\kappa$ is surjective if and only if $v:\Om_{X/S}\to \Pi\OO_X$ is surjective.
This is equivalent to the existence of an odd differential 
$$\om=\sum a_idb_i,$$
such that $\lan \om,v\ran=1$.

\noindent
(2)$\iff$(3).
Assume that we have
$$\sum a_iv(b_i)+\sum a'_jv(b'_j)=1,$$
where $(a_i)$ and $(b'_j)$ are even, $(a'_j)$ and $(b_i)$ are odd.
This implies that
$$v(\sum a_ib_i)\equiv 1\mod \NN^2,$$
so it is invertible.

Conversely, if $v(b_1)$ is invertible then $a_1v(b_1)=1$ for some $a_1$.

\noindent
(3)$\iff$(4).
Assume that $v(\th)=f$ is invertible. Then $v(f)=0$, so $v(f^{-1})=0$. This implies that
$$v(f^{-1}\th)=1.$$
The converse is clear.

\noindent
(4)$\iff$(5). Assume that we have odd $\th$ such that $v(\th)=1$.
Then for any $f$ with $v(f)=0$ we have
$$f=v(\th f).$$
The converse is clear.
\end{proof}

\begin{remark} In the case when $X$ is smooth over $S$, condition (2) from Proposition \ref{free-action-prop} is equivalent to
the condition that $v$ is everywhere non-vanishing, i.e., gives an embedding of a subbundle $\Pi\OO_X\hra T_{X/S}$.
\end{remark}

The following example involving the $\Pi$-symmetric superprojective space is important for us.

\begin{example}\label{proj-action-ex}
Let $V$ be the superspace of dimension $n+1|n+1$ equipped with a $\Pi$-symmetry $p:V\to \Pi V$, i.e.,
an odd endomorphism $p:V\to V$ such that $p^2=-\id$ (we can assume that $V$ has coordinates $(x_0,\ldots,x_n;\th_0,\ldots,\th_n)$ and $p(x_i)=\th_i$, $p(\th_i)=-x_i$).
We have a free action of $\A^{0|1}$ on $G(1|0,V)=\P^{n|n+1}$ given In homogeneous coordinates by
$$\psi(\ldots:x_i:\ldots;\ldots:\th_i:\ldots)=(\ldots:x_i+\psi\th_i:\ldots;\ldots:\th_i-\psi x_i:\ldots)$$
such that the quotient is the $\Pi$-projective space $\P^n_\Pi$. 
The corresponding $\A^{0|1}$-torsor map
$$\P^{n|n+1}\to \P^n_{\Pi}$$
sends a $1|0$-dimensional subspace $L\sub V$ to the $\Pi$-symmetric $1|1$-dimensional subspace $L+p(L)$.
\end{example}

\subsection{$v$-connections}

Let $v$ be a homological vector field on a superscheme $X$, and let $\FF$ be a quasicoherent sheaf on $X$.
A {\it $v$-connection} on $\FF$ is a map of sheaves of abelian groups $\nabla:\FF\to \Pi\FF$
satisfying
$$\nabla(fs)=v(f)s+(-1)^{|f|}f\nabla(s),$$
for a function $f$.

The curvature $c(\nabla)$ of a $v$-connection is the $\OO_X$-linear operator $\nabla^2:\FF\to \FF$.
A $v$-connection is called {\it flat} if $c(\nabla)=0$.

Note that $c(\nabla)$ commutes with $\nabla$. In particular, if $\nabla_L$ is a $v$-connection on a line bundle $L$
then $c(\nabla_L)$ can be viewed as an even function which satisfies
$$v(c(\nabla_L))=0.$$

If $\nabla$ is a $v$-connection on $\FF$ then any other $v$-connection is given by 
$$\nabla'(s)=\nabla(s)+\phi(s),$$
where $\phi$ is an $\OO$-linear map $\FF\to \Pi\FF$.
In particular, the set of $v$-connections on a line bundle $L$ (if non-empty) is a principal homogeneous space for $H^0(X,\OO_X)^-$.

It is easy to check that for a $v$-connection $\nabla$ on a line bundle $L$ and an odd global function $\phi$
one has
$$c(\nabla+\phi)=c(\nabla)+v(\phi).$$


The tensor product of $v$-connections is defined in the usual way (taking into account that $\nabla$ is odd in the sign convention).

\begin{lemma}\label{v-conn-lem}
Let $v$ be a homological vector field on $X$. 

\noindent
(i) If the map $H^1(X,\ker(v))^-\to H^1(X,\OO_X)^-$ is zero (e.g., if $H^1(X,\OO_X)^-=0$) then every line bundle on $X$ admits a $v$-connection.

\noindent
(ii) Assume that $H^0(X,\OO_X)^-=0$.
Let $L$ be a line bundle on $X$ admitting a $v$-connection $\nabla$.
Then $L$ has a unique $v$-connection. 
In particular, if $c(\nabla)\neq 0$ then $L$ does not admit a flat $v$-connection.
\end{lemma}

\begin{proof}
(i) A standard calculation shows that the obstruction to the existence of a $v$-connection on a line bundle $L$ is the image of $[L]\in H^1(X,\OO_X^{*,+})$ under the map
$$H^1(X,\OO_X^{*,+})\to H^1(X,\OO_X)^-,$$
induced by the morphism of sheaves
$$\OO_X^{*,+}\to \OO_X^-: f\mapsto v(f)f^{-1}.$$
It remains to observe that this morphism factors through $\ker(v)^-\sub \OO_X^-$ since $v(v(f)f^{-1})=0$.
Thus, if the latter embedding induces the zero map on $H^1$,
the obstruction vanishes.

\noindent
(ii) Any other $v$-connection has form $\nabla+\phi$ where $\phi$ is a global odd function on $X$, so $\phi=0$ by the assumption.
\end{proof}

\begin{lemma}\label{v-conn-equiv-structure-lem}
Let $v$ be a homological vector field on $X$. To give a flat $v$-connection $\nabla$ on a quasicoherent sheaf $\FF$ over $X$ is equivalent to 
equipping $\FF$ with an equivariant structure with respect to the corresponding $\A^{0|1}$-action on $X$.
\end{lemma}

The proof is straightforward.

From Lemma \ref{v-conn-equiv-structure-lem} one derives in a standard way (cf.\ \cite[Sec.\ 12, Thm.\ 1]{Mum-ab}) the following result.

\begin{prop}
Let $X\to Y$ be $\A^{0|1}$-torsor and let $v$ be the corresponding homological vector field on $X$.
The the category of coherent sheaves on $Y$ is equivalent to the category of coherent sheaves on $X$ with a flat $v$-connection.
\end{prop}

By the Picard group of a superscheme we mean the group of isomorphism classes of line bundles of rank $1|0$.

\begin{cor}
Let $\pi:X\to Y$ be $\A^{0|1}$-torsor and let $v$ be the corresponding homological vector field on $X$.
Consider the induced map on the Picard groups $\pi^*:\Pic(Y)\to \Pic(X)$.
Then the image of $\pi^*$ consists of the classes of line bundles on $X$ that can be equipped with a flat $v$-connection,
while the kernel of $\pi^*$ can be identified with the quotient
$$H^0(X,\ker(v)^-)/\{f^{-1}v(f) \ |\ f\in H^0(X,\OO^*)^+\}.$$
\end{cor}

\begin{proof} The statement about the image of $\pi^*$ is clear. The kernel of $\pi^*$ corresponds to isomorphism classes of flat connections on $\OO_X$.
Such connections are given by $\nabla=v+\phi$, where $\phi$ is an odd function and $v(\phi)$. Furthermore, for an invertible even function $f$, we have
$$f^{-1}(v+\phi)f=v+\phi+f^{-1}v(f).$$
\end{proof}

\subsection{Line bundles over $\A^{0|1}$-torsors, $\Pi$-symmetric $1|1$-bundles and the odd Heisenberg group}

Recall the supergroup scheme $GQ(1)\sub GL(1|1)$ (see \cite[Sec.\ 1.8.4]{BL}): its $A$-valued point is the set of $a_0+a_1\in A^*$. 
The embedding $GQ(1)\hra GL(1|1)$ is given by
$$a_0+a_1\mapsto \left(\begin{matrix} a_0 & a_1 \\ -a_1 & a_0 \end{matrix}\right)$$
(here we follow the convention that the matrix $(a_{ij})$ of a linear supertransformation $T$ is defined by $T(e_j)=\sum e_i a_{ij}$; this explains why our sign is different say
from that in \cite{MPV}).

There is a natural homomorphism 
$$GQ(1)\to\A^{0|1}: a_0+a_1\mapsto a_0^{-1}a_1$$ 
Its kernel is a central subgroup of $GQ(1)$ isomorphic to $\G_m=GL(1|0)$. Thus, we have a central extension sequence
\begin{equation}\label{GQ1-central-ext-seq}
1\to \G_m\to GQ(1)\to \A^{0|1}\to 0.
\end{equation}
Note that there is a splitting $a_1\mapsto 1+a_1$ of the projection $GQ(1)\to \A^{0|1}$ (not compatible with the group laws). 
The corresponding $2$-cocycle of $\A^{0|1}$ with values in $\G_m$ is determined from
$$(1+a)(1+b)=1+a+b+ab=(1+ab)(1+(a+b)),$$
so it is given by $c(a,b)=1+ab$.

We would like to view $GQ(1)$ as an {\it odd Heisenberg group}, a central extension of $\A^{0|1}$ by $\G_m$.
\footnote{In \cite{MPV} the group $GQ(1)$ is denoted by $\G_m$ (and what we denote by $\G_m$ is denoted by $\G_m^{1|0}$).}

\begin{prop}\label{Heis-v-conn-prop} 
Let $X$ be a superscheme with an action of $\A^{0|1}$ and let $v$ be the corresponding homological vector field.
We let $GQ(1)$ act on $X$ through the projection $GQ(1)\to \A^{0|1}$.
For a quasicoherent sheaf $\FF$ over $X$ there is a natural bijection between the set of $v$-connections $\nabla$ on $\FF$
with $c(\nabla)=n\cdot\id$, where $n\in \Z$, and weight-$n$ actions of $GQ(1)$ on $\FF$ (i.e., such that $\G_m\sub GQ(1)$ acts through the character
$\la\mapsto \la^n$).
\end{prop}

\begin{proof} This is clear if we think of the action of $GQ(1)$ as a projective action of $\A^{0|1}$: a $v$-connection $\nabla$ lifts the action of $v$ to $\FF$
and the condition $[\nabla,\nabla]=n\id$ precisely means that we get an action of the central extension with the central character $\la\mapsto \la^n$.
\end{proof}

\begin{prop}\label{GQ1-torsor-prop} 
Let $X$ be a superscheme. There is a natural equivalence between the following groupoids:
\begin{itemize}
\item $\Pi$-symmetric $1|1$-bundles on $X$;
\item data $(\wt{X}\to X, L, \si)$, consisting of an $\A^{0|1}$-torsor $\wt{X}\to X$, an even line bundle $L$ over $\wt{X}$ and a weight-$1$
action $\si$ of $GQ(1)$ on $L$ compatible with the action of $\A^{0|1}$ on $\wt{X}$.
\end{itemize}
\end{prop}

\begin{proof} From $(\pi:\wt{X}\to X,L,\si)$ we construct the $1|1$-bundle by setting $W:=\pi_*L$. The $\Pi$-symmetry is given by the action of 
$\left(\begin{matrix} 0 & 1 \\ 1 & 0\end{matrix}\right)\in GQ(1)$.
Conversely, starting from a $\Pi$-symmetric $1|1$-bundle $W$ we define $\pi:\wt{X}\to X$ is the projectivization of $W$, i.e., $\wt{X}=G(1|0;W)$.
\end{proof}



\begin{example}\label{Grassm-action-ex}
As in Example \ref{proj-action-ex} let us consider a superspace $V$ of dimension $n|n$ equipped with a $\Pi$-symmetry $p:V\to \Pi V$.
There is a natural embedding of $GQ(1)$ into $GL(V)$: an $R$-point $a_0+a_1$ of $GQ(1)$ is mapped to $a_0\cdot \id+a_1\cdot p$.
For any $(a,b)$, the corresponding action of $GQ(1)$ on the Grassmannian $G(a|b,V)$ factors through an action of $\A^{0|1}$ on $G(1|1,V)$. 
Explicitly, an $R$-point of the supergrassmannian $G(a|b,V)$ corresponds to $W\sub V_R$ and the action of $\psi\in R^-$ sends $W$ to $(\id+\psi p)(W)$.

In the case $(a,b)=(1,0)$, we recover the action of $\A^{0|1}$ on $\P^{n-1|n}$ considered in Example \ref{proj-action-ex}, such that the quotien is 
the $\Pi$-projective space $\P^{n-1}_\Pi$. 
Note that in this case there is a natural weight-$1$-action of $GQ(1)$ on the tautological line bundle $\OO(-1)$ over $\P^{n-1|n}$: the corresponding $\G_m$-torsor over $\P^{n-1|n}$
is the complement to the origin in $V$ and the action of $GQ(1)$ is induced by the embedding $GQ(1)\sub GL(V)$. The $\Pi$-symmetric $1|1$-bundle on $\P^{n-1}_\Pi=\P^{n-1|n}/\A^{0|1}$
corresponding to this $GQ(1)$-action on $\OO(-1)$ by Proposition \ref{GQ1-torsor-prop} is nothing else but the universal $\Pi$-symmetric $1|1$-bundle on $\P^{n-1}_\Pi$.

For arbitrary $(a,b)$, the embedding of $GQ(1)$ into $GL(V)$ gives a weight-$1$-action of $GQ(1)$ on $\SS$, the universal subbundle over $G(a|b,n|n)$.
The corresponding action of $GQ(1)$ on $\Ber(\SS)$ has weight $a-b$. In particular, for $b=a-1$, we get from Proposition \ref{GQ1-torsor-prop}
a $\Pi$-symmetric $1|1$-bundle over the quotient $G(a|a-1,n|n)/\A^{0|1}$. However, neither this $1|1$-bundle, nor its dual have any global sections. In fact, 
we will prove later that for $n>a>1$ and $n\ge 4$, the quotient $G(a|a-1,n|n)/\A^{0|1}$ is not $1|1$-embeddable (see Cor.\ \ref{Ga,a-1-cor}).
\end{example}


\begin{remark}
The exact sequence \eqref{GQ1-central-ext-seq} leads to a long exact sequence
$$\ldots\to H^1(X,GQ(1))\to H^1(X,\A^{0|1})\rTo{\de} H^2(X,\G_m)\to\ldots$$
The connecting map $\de$ is easy to calculate (see \cite[Sec.5]{MPV}): it is given by the composition
$$H^1(X,\A^{0|1})\rTo{c\mapsto c\cup c} H^2(X,\G_a)\rTo{\exp} H^2(\G_m)$$
(actually one can make sense of the composition in any characteristic).
For example, for $X=\P^n_\Pi$ over an algebraically closed field $k$, 
we have identifications $H^1(X,\A^{0|1})\simeq k$, $H^2(X,\G_m)\simeq k/\Z$, such that the map $\de$ is given by 
$x\mapsto x^2\mod \Z$.
\end{remark}

\subsection{$\A^{0|1}$-fibrations}

It is easy to see that all automorphisms of $\A^{0|1}$ are affine transformations $\th\mapsto a\th+\psi$, where $a$ is an even invertible parameter and $\psi$ is an odd parameter.

\begin{prop}\label{A01-fibrations-prop}
There is an equivalence between the groupoid $\A^{0|1}$-fibrations over $X$ and that of the following data: an even line bundle $L$ on $X$
together with a $\Pi L$-torsor over $X$. The latter groupoid can be also thought of as that of $1|1$-vector bundles $V$ on $X$ equipped with an 
embedding $\OO_X\to V$ such that $V/\OO_X$ is a line bundle of rank $0|1$.
\end{prop}

\begin{proof}
The bijection of the corresponding isomorphism classes follows from the general yoga of noncommutative cohomology: we have an exact sequence of group schemes
$$0\to \A^{0|1}\to \Aut(\A^{0|1})\to \G_m\to 1,$$
admitting a (non-normal) splitting $\G_m\to \Aut(\A^{0|1})$.
This implies that the corresponding map $H^1(X,\Aut(\A^{0|1}))\to H^1(X,\G_m)$ is surjective and the fiber over
$[c]\in H^1(X,\G_m)$ is identified with $H^1(X,\A^{0|1}_c)$, where $\A^{0|1}_c$ is the twist of $\A^{0|1}_X$ by the $1$-cocycle $c$.
Thus, if $[c]$ corresponds to a line bundle $L$ then $\A^{0|1}_c=\Pi L$.

The functor between the corresponding groupoids is defined as follows.
Given an $\A^{0|1}$-fibration $\pi:\wt{X}\to X$ we define the line bundle $L$ on $X$ by
$$\Pi L^{-1}:=\pi_*\OO_{\wt{X}}/\OO_X.$$
Furthermore, the set of $\OO_X$-linear splittings of the exact sequence
$$0\to \OO_X\to \pi_*\OO_{\wt{X}}\to \Pi L^{-1}\to 0$$
is a $\Pi L$-torsor over $X$.

Conversely, given an extension of supervector bundles
$$0\to \OO_X\to V\to M\to 0\to 0$$
where $M$ has rank $0|1$, there is a unique structure of $\OO_X$-algebra on $V$ defined uniquely by the conditions that the $\OO_X$-module structure on $M$
is induced by the mutliplication in $V$ and that for any local odd section $s$ of $V$ one has
$s^2=0$. Indeed, locally we can choose an odd section $\th$ of $V$ projecting to a generator of $M$. Then the natural map $\OO_X[\th]\to V$ is an isomorphism of
algebras. These structures glue into a global structure of $\OO_X$-algebra on $V$.
Hence, we have $V=\pi_*\OO_{\wt{X}}$ for a canonically defined $\A^{0|1}$-fibration $\wt{X}\to X$.
\end{proof}



\section{Quotients by free $\A^{0|n}$-actions}

\subsection{Existence of quotients by a free action}

Everywhere in this subsection we fix a base superscheme $S$ and work with superschemes over $S$.
We use the obvious superanalogs of the basic notions about group actions. Recall that an action of a subgroup scheme $G$ on a subsperscheme $X$ (over $S$)
is called ${\it free}$ if the morphism $\a_{G,X}:G\times_S X\to X\times_S X: (g,x)\mapsto (gx,x)$ is a closed embedding.

\begin{theorem}\label{free-quot-thm} 
Let $\A^{0|n}_S$ act freely on a superscheme $X$ over $S$. Then there exists a categorical quotient $\pi:X\to Y=X/\A^{0|n}$ for this action.
Furthermore, $\pi$ is an $\A^{0|n}$-torsor with respect to Zariski topology. If $X$ is of finite type (resp., smooth) over $S$ then so is $Y$.
\end{theorem}

The proof is based on the following result.

\begin{lemma}\label{affine-odd-derivation-lem} 
Let $v$ be an odd derivation of a superring $A$ such that $v^2=0$ and there exists an odd element $\th$ such that $v(\th)=1$. Set $A^v=\ker(v:A\to A)$.
Then the natural ring homomorphism
$$A^v[\th]\to A$$
is an isomorphism (where on the left we view $\th$ as an independent variable).
If $A$ is finitely generated as an algebra over a sub-superring $C\sub A^v$ then so is $A^v$.
\end{lemma}

\begin{proof}
We have to prove that every element of $A$ can be written uniquely in the form $a+b\th$ with $a,b\in A^v$.
The fact that $a$ and $b$ can be recovered from $a+b\th$ follows immediately from the identity
$$v(a+b\th)=\pm b.$$
On the other hand, we have
$$a=v(a\th)\pm v(a)\th,$$
where $v(a\th)$ and $v(a)$ are in $A^v$.

For the last assertion, suppose $x_i$ is a finite set of generators of $A$ over $C$. Write each $x_i$ in the form $a_i+b_i\th$ with $a_i,b_i\in A^v$.
Then it is easy to see that $(a_i)$ and $(b_i)$ generate $A^v$ over $C$.
\end{proof}

We also need the following general statement.

\begin{lemma}\label{subgroup-quot-lem}
Let $G$ be a supergroup scheme acting freely on a superscheme $X$, and let $H\sub G$ be a normal subgroup, flat over $S$, such that the quotient $G/H$ exists as a supergroup scheme and the sequence
$$1\to H(T)\to G(T)\to (G/H)(T)\to 1$$
is exact for any $S$-superscheme $T$.
Assume that the categorical quotient $X'=X/H$ exists as a superscheme and the map $X\to X'$ is an $H$-torsor (in fppf topology).
Then

\noindent
(i) $G'=G/H$ acts freely on $X'$.

\noindent
(ii) If the categorical quotient $X'/G'$ exists then the map $X\to X'\to X'/G'$ gives a categorical quotient $X/G$.

\noindent
(iii) Under the assumption of (ii), if in addition the map $X\to X'$ is a Zariski $H$-torsor and the map $X'\to X'/G'$ is a Zariski $G'$-torsor,
then the map $X\to X'/G'=X/G$ is a Zariski $G$-torsor.
\end{lemma}

\begin{proof}
(i) It is enough to prove that the diagram
\begin{diagram}
G\times_S X &\rTo{\a_{G,X}}& X\times_S X\\
\dTo{}&&\dTo{}\\
G'\times_S X' &\rTo{\a_{G',X'}}& X'\times_S X'
\end{diagram}
is cartesian, where the vertical maps are the natural projections.
It is clear that the diagram
\begin{diagram}
G'\times_S X &\rTo{(\a',p_X)}& X'\times_S X\\
\dTo{}&&\dTo{}\\
G'\times_S X' &\rTo{\a_{G',X'}}& X'\times_S X'
\end{diagram}
is cartesian, where $\a':G'\times_S X\to X'$ is induced by $(g,x)\mapsto gx$. Hence, it is enough to show that the diagram
\begin{equation}\label{G-G'-X-X'-diagram}
\begin{diagram}
G\times_S X &\rTo{(g,x)\mapsto gx}& X\\\
\dTo{}&&\dTo{}\\
G'\times_S X &\rTo{\a'}& X'
\end{diagram}
\end{equation}
is cartesian. It is enough to show this after the faithfully flat base change $G\times_S X\to G'\times_S X$. Since the natural diagram
\begin{diagram}
H\times_S G &\rTo{(h,g)\mapsto hg}& G\\
\dTo{p_G}&&\dTo{}\\
G&\rTo{}& G'
\end{diagram}
is cartesian, it remains to check that the diagram
\begin{diagram}
H\times_S G\times_S X &\rTo{(h,g,x)\mapsto hgx}& X\\
\dTo{p_{G\times X}}&&\dTo{}\\
G\times_S X &\rTo{}& X'
\end{diagram}
is cartesian, where the bottom arrow is the composition of $(g,x)\mapsto gx$ with the projection $X\to X'$.
But we can present it as the composition of two cartesian squares
\begin{diagram}
H\times_S G\times_S X &\rTo{(h,g,x)\mapsto (h,gx)}& H\times X &\rTo{(h,x)\mapsto hx}& X \\
\dTo{p_{G\times X}}&&\dTo{p_X} && \dTo{}\\
G\times_S X&\rTo{}& X &\rTo{}& X'
\end{diagram}
where the right square is cartesian due to the fact that $X\to X'$ is an $H$-torsor.

\noindent
(ii) This is a straightforward consequence of the definitions.

\noindent
(iii) Combining Zariski local sections of the projections $X\to X'$ and $X'\to X'/G'$, we obtain that $X\to X'/G'$ admits Zariski local sections.
Hence, it remains to show that the diagram
\begin{diagram}
G\times_S X &\rTo{(g,x)\mapsto gx}& X\\\
\dTo{p_X}&&\dTo{}\\
X &\rTo{}& X'/G'
\end{diagram}
is cartesian. We have
$$X\times_{X'/G'} X\simeq (X\times_{X'/G'} X')\times_{X'} X.$$
Since $X'\to X'/G'$ is a $G'$-torsor, we have a natural isomorphism $G'\times _SX'\rTo{\sim} X'\times_{X'/G'} X'$, which induces an isomorphism
$$G'\times_S X\rTo{\sim} X\times_{X'} (X'\times_{X'/G'} X')\simeq X\times_{X'/G'} X'.$$
Now the assertion follows from the cartesian square \eqref{G-G'-X-X'-diagram}.
\end{proof}

\begin{proof}[Proof of Theorem \ref{free-quot-thm}]
We use induction on $n$.

\noindent
{\bf Case $n=1$}.
Let $v$ be the homological vector field on $X$ corresponding to this action. Let $\OO^v_X\sub \OO_X$ denote the subsheaf of $v$-horizontal functions. 
By Proposition \ref{free-action-prop}, locally there exists an odd function $\th$ on $X$ such that $v(\th)=1$. Suppose
$\Spec(A)$ is an affine open in $X$ with such a function $\th$. We claim that for any even $f\in A$, the ring $(A[f^{-1}])^v$ is naturally a localization of $A^v$.
Indeed, this is clear if $v(f)=0$. Now we observe that the localization by $f$ does not change if we replace $f$ by
$$\wt{f}:=v(f\th)=v(f)\th+f,$$
where $v(\wt{f})=0$. This easily implies that $Y:=(|X|,\OO^v_X)$ is a superscheme. The fact that $X\to Y$ is a categorical quotient immediately reduces to the affine case,
and in the affine case it is clear that the categorical quotient boils down to passing from a ring $A$ to the subring $A^v$.

By Lemma \ref{affine-odd-derivation-lem}, we also see that the map $X\to Y$ is an $\A^{0|1}$-torsor.

\noindent
{\bf Case $n>1$.} Consider a subgroup $\A^{0|n-1}\sub \A^{0|n}$. By the induction assumption, there exists a quotient $X'=X/\A^{0|n-1}$. Furthermore, by Lemma \ref{subgroup-quot-lem},
$\A^{0|1}= \A^{0|n}/\A^{0|n-1}$ acts freely on $X'$, and the quotient $X'/\A^{0|1}$, which we know exists, is also the quotient $X/\A^{0|n}$. The fact that $X\to X/\A^{0|n}$ is an $\A^{0|n}$-torsor
also follows from Lemma \ref{subgroup-quot-lem}.
\end{proof}

\begin{remark}
It is easy to see that the condition that $\A^{0|1}$-action is free is necessary in Theorem \ref{free-quot-thm}.
For example, for the homological vector field $v=\th \partial_z$ on $\A^{1|1}_k$ with coordinatees $(z,\th)$, the subring of $v$-invariants is $k+\th\cdot k[z]\sub k[z,\th]$, which is not
finitely generated over $k$.
\end{remark}

\subsection{The action on supergrassmannians}

Recall that in Example \ref{Grassm-action-ex} we defined an action of $\A^{0|1}$ on supergrassmannians $G(a|b,n|n)$
associated with a $\Pi$-symmetry on a supervector space $V$ of superdimension $n|n$.

\begin{lemma} The above action of $\A^{0|1}$ on $G(a|b,n|n)$ is free if and only if $a\neq b$.
\end{lemma}

\begin{proof}
In the case $a=b$ we have a nonempty subscheme consisting of $W$ such that $p(W)=W$. The action of $\A^{0|1}$ on this subscheme is trivial.

To prove freeness in the case $a\neq b$ we have to prove that the corresponding homological vector field is non-vanishing everywhere.
Consider the covering of $G(a|b,n|n)=G(a|b,V)$ by open affine cells associated with decompositions $V=\La_1\oplus \La_2$ into coordinate subspaces,
where $\La_1$ has dimension $a|b$, so that subspaces in the corresponding cell are the graphs of maps $\La_1\to \La_2$.

Let $p_{ij}:\La_i\to \Pi\La_j$ be the components of the $\Pi$-symmetry $p:V\to \Pi V$. The action of $\A^{0|1}$ is given by $W\mapsto (\id+\psi p)(W)$, so
if $W=(\id+f)(\La_1)$ for some $f:\La_1\to \La_2\ot R$ (where $R$ is a test superring), then it maps to 
$$(\id_V+\psi p)(\id_{\La_1}+f)(\La_1)=(\id_{\La_1}+f')(\La_1) \ \text{ with}$$
$$f'=[f+\psi(p_{21}+p_{22}f)]\cdot [1+\psi(p_{11}+p_{12}f)]^{-1}=f+\psi[p_{21}+p_{22}f-fp_{11}-fp_{12}f]=0.$$
It remains to prove that there does not exist $f$ such that
$$p_{21}+p_{22}f-fp_{11}-fp_{12}f=0.$$
Indeed, composing with $p_{12}$ we get
$$-\id_{\La_1}-p_{12}fp_{11}-p_{12}fp_{12}f=0,$$
$$(p_{12}f)(p_{11}+p_{12}f)=-\id_{\La_1}.$$
But $p_{12}f$, viewed as a map from $\Pi\La_1$ to $\La_1$, cannot be surjective since $a\neq b$, so we get a contradiction.
\end{proof}



\section{$\Pi$-projectivity}

\subsection{Criterion for $\Pi$-projectivity}

\begin{lemma} For every $n$, there exists a closed embedding $\P^{n-1|n}\hra \Pi^{2n-1}_\Pi$. Hence, every projective superscheme is $\Pi$-projective.
\end{lemma}

\begin{proof} Given a supervector space $V$, we can equip the space $V\oplus \Pi V$ with a natural $\Pi$-symmetry.
Then to every family of $1|0$-dimensional subspaces $L\sub V$ we associate naturally a family of $\Pi$-symmetric $1|1$-subspaces $L\oplus \Pi L\sub V\oplus \Pi V$.
\end{proof}

\begin{theorem}\label{Pi-proj-thm}
(i) If $X$ is $\Pi$-projective then there exists an $\A^{0|1}$-torsor over $X$ whose total space is projective.

\noindent
(ii) Assume that the base is $\Spec(k)$ where $k$ is an algebraically closed field.
Assume that $H^0(X,\OO)^+=k$ and there exists an $\A^{0|1}$-torsor over $X$ whose total space $\wt{X}$ is projective and such that the morphism
$H^1(X,\OO_X)^-\to H^1(\wt{X},\OO_{\wt{X}})^-$ is zero. Then $X$ is $\Pi$-projective.
\end{theorem}

\begin{proof} 
(i) If $X$ is embedded into $\P^{n-1}_\Pi$ then we can consider the pullback to $X$ of the standard $\A^{0|1}$-torsor $\P^{n-1|n}\to \P^{n-1}_\Pi$. Its total space embeds into $\P^{n-1|n}$.

\noindent
(ii) Let $L$ be a very ample line bundle on $\wt{X}$ and let $v$ be the homological vector field on $\wt{X}$ corresponding to the $\A^{0|1}$-action. 
Since $\ker(v)=\OO_X\sub \OO_{\wt{X}}$, Lemma \ref{v-conn-lem}(i) implies that $L$ admits a $v$-connection $\nabla$.
Next, we look at the curvature $c(\nabla)$.

\noindent
{\bf Case 1}: $c(\nabla)\neq 0$. Then we can rescale $v$, so that $c(\nabla)=1$. Let $s_1,\ldots,s_n$ be the basis of global sections of $L$. Then
using $s_1,\ldots,s_n,\nabla(s_1),\ldots,\nabla(s_n)$ get a $\A^{0|1}$-equivariant embedding of $\wt{X}$ into $\P^{n-1|n}$. Passing to
quotients get an embedding of $X$ into $\P^{n-1}_\Pi$.

\noindent
{\bf Case 2}: $c(\nabla)=0$. Then $L$ descends to $X$. Since its restriction to $X_{red}$ is very ample, it follows that $X$ is projective (see e.g., \cite[Prop.\ A.2]{FKP}), hence it is $\Pi$-projective.
\end{proof}

\begin{example} Consider $\P^{m-1|m}\times \P^{n-1|n}$ with diagonal $\A^{0|1}$-action. Then by Theorem \ref{Pi-proj-thm}(ii), the quotient is $\Pi$-projective. On the other hand, this quotient is not projective
for $m\ge 3$, $n\ge 3$, as we can find an embedding $\P^{2|3}\to \P^{m-1|m}\times \P^{n-1|n}$ compatible with $\A^{0|1}$-actions, which implies that the quotient contains $\P^2_\Pi$ as a sub-superscheme.
\end{example}

\begin{example} Consider $G(a|0,n|n)$ with the $\A^{0|1}$-action defined in Example \ref{Grassm-action-ex}. Note that $G(a|0,n|n)$ is projective since the restriction of $\Ber(\UU)$ to the reduced space
$G(a,n)$ is ample (see \cite[Prop.\ A.2]{FKP}). It is easy to check $H^1(G(a|0,n|n),\OO)=0$ for $n\ge a+2$. Namely, one checks this by induction on $a$
using the fibration $F(a-1|0,a|0,n|n)\to G(a|0,n|n)$. Hence, by Theorem \ref{Pi-proj-thm}(ii), the quotient $G(a|0,n|n)/\A^{0|1}$ is $\Pi$-projective.
\end{example}


\subsection{Criterion for non-$\Pi$-projectivity}

The following criterion for non-$\Pi$-projectivity is similar to the one used in \cite[Sec.\ 7]{CNR}.

\begin{prop} If $X$ is non-projective and $H^1(X,\OO)^-=0$ then $X$ is not $\Pi$-projective.
\end{prop}

\begin{proof}
Suppose $X$ is $\Pi$-projective. Then there exists an $\A^{0|1}$-torsor $\wt{X}\to X$ such that $\wt{X}$ is projective. But if $H^1(X,\OO)^-=0$ then
there exists a section $X\to \wt{X}$, so $X$ is also projective, which is a contradiction.
\end{proof}

\begin{example} Using this criterion we can easily check that $G(1|1,2|2)$ is not $\Pi$-projective. Hence, $G(1|1,n|n)$ is not $\Pi$-projective for $n>1$.
\end{example}

\section{$1|1$-embeddability}

\subsection{General criteria for non-$1|1$-embeddability}\label{11-emb-gen-sec}

\begin{prop}\label{1|1-emb-fibr-prop}
If $X$ is $1|1$-embeddable then there exists an $\A^{0|1}$-fibration over $X$ whose total space is projective.
\end{prop}

\begin{proof}
If $X$ is embedded into $G(1|1,n|n)$, then the total space of the pullback of the $\A^{0|1}$-fibration $F(1|0,1|1,n|n)\to G(1|1,n|n)$
embeds into $F(1|0,1|1,n|n)$ which is projective.
\end{proof}

\begin{cor}\label{1|1-emb-fibr-cor}
Assume that $\Pic(X)=0$. If $X$ is $1|1$-embeddable then there exists an $\A^{0|1}$-torsor over $X$ with projective total space.
\end{cor}

\begin{proof} Indeed, by Proposition \ref{A01-fibrations-prop}, if $\Pic(X)=0$ then every $\A^{0|1}$-fibration over $X$ is an $\A^{0|1}$-torsor.
\end{proof}

\begin{prop}\label{non-1|1-emb-prop-1}  Let $X$ be a proper superscheme over $k$ of positive dimension.

\noindent
(i) Assume that $\Pic(X)=0$ and $H^1(X,\OO)^-=0$. Then $X$ is not $1|1$-embeddable. 

\noindent
(ii) Assume that $H^1(X,L)^-=0$ for every even line bundle $L$ on $X$ and $X$ is not projective. Then $X$ is not $1|1$-embeddable.
\end{prop}

\begin{proof}
(i) If $X$ were $1|1$-embeddable we would deduce from Corollary \ref{1|1-emb-fibr-cor}
that there exists an $\A^{0|1}$-torsor $\wt{X}\to X$ with projective $\wt{X}$.
But any such torsor is trivial since $H^1(X,\OO)^-=0$. Hence, we would get that $X$ is projective
which is impossible if $X$ has trivial $\Pic$.

\noindent
(ii) By Proposition \ref{A01-fibrations-prop}, any $\A^{0|1}$-fibration $\wt{X}\to X$
has a structure of a $\Pi L$-torsor, for some even line bundle $L$ on $X$. If $H^1(X,L)^-=0$ then this torsor is necessarily trivial,
hence, there exists a section $X\to \wt{X}$. By Proposition \ref{1|1-emb-fibr-prop}, if $X$ were $1|1$-embeddable then there would exist
an $\A^{0|1}$-fibration with projective $\wt{X}$. Since $X$ is embedded into $\wt{X}$, it would be projective which is a contradiction.
\end{proof}

\begin{remark}
It is instructive to take $X=G(1|1,2|2)$ and to see which assumptions of Proposition \ref{non-1|1-emb-prop-1} are not satisfied in this case.
We know that $X$ is not projective and we have 
$$H^1(X,\OO)^-=H^1(\P^1\times \P^1, (\OO(-1)\boxtimes \OO(-1))^{\oplus 2})=0.$$
However, $\Pic(X)\neq 0$: we have a nontrivial line bundle $L=\Ber(\SS)$, where $\SS$ is the tautological bundle of rank $1|1$. We have $L|_{\P^1\times \P^1}=\OO(-1)\boxtimes\OO(1)$.
Hence,
$$H^1(X,L)^-=H^1(\P^1\times\P^1, (\OO(-2)\boxtimes \OO)^2)\neq 0.$$
\end{remark}

\begin{example}\label{CY22-ex} 
For any smooth variety $X_0$ over a field $k$, a vector bundle $V$, and a class $e\in H^1(X_0,T_{X_0}\ot \bigwedge^2 V)$, we can define a superscheme
$X=X(X_0,V,e)$ as follows. Let $\OO_X^+\to \OO_{X_0}$ be the square zero extension of $\OO_{X_0}$ by $\bigwedge^2 V$ with the class $e$,
and set $\OO_X^-=V$. The multiplication of $\OO_X^+$ on $\OO_X^-$ is given by the $\OO_{X_0}$-module structure on $V$ via the projection $\OO_X^+\to \OO_{X_0}$, while
the multiplication 
$$\OO_X^-\times \OO_X^-\to \bigwedge^2 V\sub \OO_X^+$$ 
is given by the wedge product.

Now let us specialize to the case when $X_0=S$ is a smooth projective surface, $V$ has rank $2$ and $\bigwedge^2 V\simeq \om_S$. 
Then $X=X(S,V,e)$ is a smooth CY supervariety of dimension $2|2$ (this construction is used in \cite{CNR} in the case $S=\P^2$).

Assume now that in addition $\Pic(S)=\Z$, $H^1(S,V)=0$, $k$ has characteristic zero, and
$$e=c_1(L)\in H^1(S,\Om^1_S)\simeq H^1(S,T_S\ot \om_S)$$
where $L$ is an ample line bundle on $S$.
Then Proposition \ref{non-1|1-emb-prop-1} implies that $X=X(S,V,c_1(L))$ is not $1|1$-embeddable. Indeed, by assumption we have $H^1(X,\OO_X^-)=0$. The triviality of $\Pic(X)$
is established similarly to \cite[Thm.\ 5.2]{CNR}: we use the exact sequence
$$1\to 1+\om_S\to (\OO_X^+)^*\to \OO_S^*\to 1$$
and observe that the connecting homomorphism $H^1(S,\OO_S^*)\to H^2(S,\om_S)$ sends the class of a line bundle $M$ to $c_1(M)\cup c_1(L)$. Since $\Pic(S)=\Z$, this map
is injective, so no line bundle on $S$ extends to $\OO_X^+$. 
\end{example}



One can try to relax the condition of vanishing of $H^1(X,\OO)^-$ in Proposition \ref{non-1|1-emb-prop-1}. We give an example of this in the situation when we have an $\A^{0|1}$-torsor
$\wt{X}\to X$. Note that for such a torsor the exact sequence of sheaves
$$0\to \OO_X\to \OO_{\wt{X}}\rTo{v} \Pi\OO_X\to 0$$
gives rise to a long exact sequence
\begin{equation}\label{H1-torsor-long-ex-seq}
\ldots\to H^0(X,\OO_{\wt{X}})^-\rTo{v} H^0(X,\OO_X)^+\to H^1(\OO_X)^-\to H^1(\OO_{\wt{X}})^-\to\ldots
\end{equation}

\begin{prop}\label{non-1|1-emb-prop-2}
Let $\wt{X}\to X$ be an $\A^{0|1}$-torsor. Assume that $\Pic(X)=0$, $H^0(X,\OO_X)=k$ (of characteristic zero),
 the space $H^1(X,\OO_X)^-$ has dimension $\le 1$ (e.g., $H^1(\wt{X},\OO_{\wt{X}})^-=0$), and $\wt{X}$ is non-projective. Then
$X$ is not $1|1$-embeddable.
\end{prop}

\begin{proof} First, we see that any $\A^{0|1}$-fibration over $X$ is an $A^{0|1}$-torsor. They are classified by $H^1(\OO_X)^-$. All come from $H^0(\OO_X^+)=k$. 
So would get that $\wt{X}$ is projective.
Note that the exact sequence \eqref{H1-torsor-long-ex-seq} implies that if $H^1(\wt{X},\OO_{\wt{X}})^-=0$ then $H^1(X,\OO_X)^-$ is at most $1$-dimensional.
\end{proof}

\subsection{More examples}

Recall that in Example \ref{CY22-ex} we considered CY supervarieties of dimension $2|2$ of the form $X=X(S,V,e)$, where $V$ is a rank $2$ vector bundle on a smooth projective
surface $S$,
such that $\bigwedge^2 V\simeq \om_S$ and $e\in H^1(S,\Om^1_S)$ is a nonzero class. 
In the case $S=\P^2$ it was proved in \cite{CNR} that every $X(\P^2,V,e)$ is embeddable into some supergrassmannian. For other surfaces $S$ the question of embeddability of
$X(S,V,e)$ is open. The following result was suggested by Emile Bouaziz.

\begin{prop}\label{CY22-prop}
(i) Assume $V=\Om^1_S$. Then for any ample line bundle $L$ on $S$, the supervariety $X(S,V,c_1(L))$ is $\Pi$-projective.

\noindent
(ii) For any nontrivial $\A^{0|1}$-torsor $\wt{X}\to X(S,V,e)$, the total space $\wt{X}$ is projective. In particular, if $H^1(S,V)\neq 0$ then
$X(S,V,e)$ is embeddable into some supergrassmannian.
\end{prop}

\begin{proof}
(i) Let us specialize the construction of the superscheme $X(X_0,V,e)$ (see Example \ref{CY22-ex}) to the case $V=\Om^1_{X_0}$ and 
$e\in H^1(X_0,T_{X_0}\ot \Om^2_{X_0})$ comes from a class $c\in H^1(X_0,\Om^1_{X_0})$ via the natural map 
$$\Om^1_{X_0}\to \Hom(\Om^1_{X_0},\Om^2_{X_0})\simeq T_{X_0}\ot \Om^2_{X_0}$$
induced by the wedge product. Let us call the resulting superscheme $X(X_0,c)$. 
It is easy to see that if $f:Y_0\to X_0$ is a morphism, and the class $c'\in H^1(Y_0,\Om^1_{Y_0})$ is obtained from $c\in H^1(X_0,\Om^1_{X_0})$
via the natural restriction map $f^*:H^1(X_0,\Om^1_{X_0})\to H^1(Y_0,\Om^1_{Y_0})$ then $f$ extends to a morphism
$\wt{f}:X(Y_0,c')\to X(X_0,c)$. Furthermore, if $f$ is a closed embedding then so is $\wt{f}$.

It is well known that for $X_0=\P^n$ and $c=c_1(\OO(1))$, the superscheme $X(\P^n,c_1(\OO(1)))$ is the truncation of the $\Pi$-projective space $\P^n_{\Pi}$,
i.e., obtained by replacing the structure sheaf with $\OO_{\P^n_{\Pi}}/\NN^3$, where $\NN$ is the ideal generated by odd functions (see e.g., \cite{Noja-Pi}).
In particular, $X(\P^n,c_1(\OO(1)))$ is $\Pi$-projective.

We are interested in the superscheme $X(S,c_1(L))$ for an ample line bundle $L$ on $S$. Replacing $L$ by a high power leads to rescaling $c_1(L)$,
hence, does not change this superscheme up to an isomorphism. Hence, we can assume that $L=i^*\OO(1)$ for an embedding $i:S\hra \P^n$.
Then $X(S,c_1(L))$ embeds into $X(\P^n,c_1(\OO(1)))$, hence it is $\Pi$-projective.

\noindent
(ii) Let us fix a nonzero class $\a\in H^1(S,V)$, and let $\wt{X}\to X=X(S,V,e)$ be the corresponding $\A^{0|1}$-torsor. Then $\OO_{\wt{X}}^+$ is 
an extension of $\OO_S$ by a square zero ideal $\II\sub \OO_{\wt{X}}^+$. Furthermore, since $\OO_X^+\sub \OO_{\wt{X}}^+$, we have an exact sequence
of $\OO_S$-modules
\begin{equation}\label{V-omS-ext}
0\to \om_S\to \II\to V\to 0
\end{equation}
and the class $\wt{e}\in H^1(S,T_S\ot \II)$ defining the extension $\OO_{\wt{X}}^+\to \OO_S$ is the image of $e$ under the natural map $H^1(S,T_S\ot \om_S)\to H^1(S,T_S\ot \II)$.
Furthermore, it is easy to see that the class $\a'\in \Ext^1(V,\om_S)$ of the extension \eqref{V-omS-ext} is the image of $\a$ under the isomorphism
$\Ext^1(V,\om_S)\simeq H^1(S,V^\vee\ot\om_S)\simeq H^1(S,V)$.

By \cite[Prop.\ A.2]{FKP}, to prove that $\wt{X}$ is projective, it is enough to show that some (even) ample line bundle on $S$ extends to a line bundle on $\wt{X}$. The standard argument using 
the exact sequence 
$$0\to 1+\II\to (\OO_{\wt{X}}^+)^*\to \OO_S^*\to 1$$
shows that the obstacle to extending a line bundle $M$ on $S$ to $\wt{X}$ is given by the product $c_1(M)\cup \wt{e}\in H^2(S,\II)$. This means that this obstacle is the image of
the product $c_1(M)\cup e\in H^2(S,\om_S)$ under the natural map $H^2(S,\om_S)\to H^2(S,\II)$. We claim that the latter map is zero.
Indeed, the exact sequence of cohomology associated with \eqref{V-omS-ext} shows that it is enough to prove surjectivity of the connecting homomorphism
$$H^1(S,V)\to H^2(S,\om_S)$$
which is given by the cup product with $\a'\in \Ext^1(V,\om_S)$. But this immediately follows from nonvanishing of $\a'$ and the fact that the Serre duality pairing
$$H^1(S,V)\ot \Ext^1(V,\om_S)\to H^2(S,\om_S)$$
is nondegenerate.

Hence, if $H^1(X,V)\neq 0$ then there exists an $\A^{0|1}$ torsor $\wt{X}\to X$ such that $\wt{X}$ is projective, and by \cite[Prop.\ 4.28]{BHRP}, $X$ is embeddable into some supergrassmannian.
\end{proof}

\begin{prop}\label{non-embed-G21-prop} 
Assume the characteristic is zero.
Let us consider the free $\A^{0|1}$-action on $X=G(2|1,4|4)$ defined in Example \ref{Grassm-action-ex}. Then the quotient $Y=G(2|1,4|4)/\A^{0|1}$ is not $1|1$-embeddable. 
\end{prop}

\begin{lemma}\label{H1-Gr-lem}
(i) For $n\ge 2$ and $X=G(1|1,n|n)$ one has $H^1(X,\OO)=H^1(X,\NN)=0$.

\noindent
(ii) For $n\ge 4$ and $X=G(2|1,n|n)$ one has $H^1(X,\OO)=H^1(X,\NN)=0$.

\noindent
(iii) The Picard group of $X=G(2|1,n|n)$ is isomorphic to $\Z$, with $\Ber(\SS)$ as a generator, where $\SS$ is the universal subbundle.
\end{lemma}

\begin{proof}
(i) The bosonic truncaton of $X$ is $X_0=\P^{n-1}\times \P^{n-1}$ and 
$$\bigoplus \NN^i/\NN^{i+1}\simeq {\bigwedge}^\bullet(\OO(-1)\boxtimes\Om^1(1)\oplus \Om^1(1)\boxtimes\OO(-1))\simeq
\bigoplus_{i,j}\Om^j(j-i)\boxtimes \Om^i(i-j)$$
(see \cite[4.3.15]{Manin}). Now the calculation of cohomology of $\Om^i(m)$ on projective spaces implies that $H^1(X_0,\NN^i/\NN^{i+1})=0$,
and the result follows.

\noindent
(ii) Let us consider the $\P^{1|1}$-fibration
$$p:F=F(1|1,2|1,n|n)\to G(2|1,n|n)=X.$$
Since $H^*(\P^{1|1},\OO)=k$, using base change we get $Rp_*\OO_F\simeq \OO_X$.
Thus, it is enough to prove that $H^1(F,\OO_F)=0$. Now we use the $\P^{n-2|n-1}$-fibration
$$q:F\to G(1|1,n|n)=G.$$
Since $R^0q_*\OO_F=\OO_G$ and $R^1q_*\OO_F=0$, we get
$H^1(F,\OO_F)=H^1(G,\OO_G)$. It remains to use part (i).

\noindent
(iii) The Picard group of the corresponding reduced space $X_0=G(2,n)\times \P^{n-1}$ is isomorphic to $\Z\times \Z$.
Furthermore, due to (ii), the restriction map $\Pic(X)\to \Pic(X_0)$ is injective. It remains to show that $\OO(m)\boxtimes\OO(n)$ does not extend
to a line bundle on $X$ unless $m=-n$ (where $\OO(-1)$ on $G(2,n)$ is the determinant of the universal subbundle).
Restricting to $G(1|1,2|2)\sub G(2|1,n|n)$ we reduce to the similar question for $G(1|1,2|2)$ which is well known and can be analyzed easily
(see \cite[Ch.\ 4.\S 3]{Manin}).
\end{proof} 

\begin{remark} It is known that the Picard group of any supergrassmannian is generated by the Berezinian of the universal subbundle
(see \cite{PS}).
\end{remark}

\begin{proof}[Proof of Proposition \ref{non-embed-G21-prop}]
First, we claim that $\Pic Y=0$. By Lemma \ref{H1-Gr-lem}(iii), every line bundle on $X=G(2|1,4|4)$ is isomorphic to $\Ber(\SS)^n$ for some $n\in \Z$, where
$\SS$ is the universal subbundle. It is enough to check that $\Ber(\SS)^n$ does not admit a flat $v$-connection for $n\neq 0$. 
Recall that we have a natural $GQ(1)$-equivariant structure on $\Ber(\SS)$ of weight $2-1=1$ (see Example \ref{Grassm-action-ex}).
Hence, by Proposition \ref{Heis-v-conn-prop}, there is a $v$-connection $\nabla$ on $\Ber(\SS)$ with curvature $1$. The induced connection on 
$\Ber(\SS)^n$ will have curvature $n$. Since $H^0(X,\OO)^-=0$, by Lemma \ref{v-conn-lem}(ii), this implies that $\Ber(\SS)^n$ does not admit a flat $v$-connection
for $n\neq 0$. 

Since $H^1(X,\OO)=0$ by Lemma \ref{H1-Gr-lem} and $X$ is not projective, using Proposition \ref{non-1|1-emb-prop-2}, we get that $Y$ is not $1|1$-embeddable.
\end{proof}

\begin{cor}\label{Ga,a-1-cor}
For $n>a>1$ and $n\ge 4$, the quotient $G(a|a-1,n|n)/\A^{0|1}$ is not $1|1$-embeddable.
\end{cor}

\begin{proof}
We can the present the superspace $V$ with the $\Pi$-symmetry $p$ as the direct sum $V_1\oplus V_2$, where both summands are preserved by $p$ and
the dimension of $V_1$ is $4|4$. Then fixing a subspace $W_2\sub V_2$ of dimension $a-2|a-2$, preserved by $p$, we get a closed embedding 
$$G(2|1,V_1)\hra G(a|a-1,V): W_1\mapsto W_1\oplus W_2,$$
compatible with $\A^{0|1}$-actions. This embedding induces a closed embedding of the quotients, so the assertion follows from Proposition \ref{non-embed-G21-prop}.
\end{proof}


\begin{prop}\label{prod-Pi-proj-planes-prop} 
$\P^2_\Pi\times \P^2_\Pi$ is not $1|1$-embeddable.
\end{prop}

\begin{lemma}\label{prod-Pi-proj-planes-tors-lem} 
The total space of any $A^{0|1}$-torsor over $\P^2_\Pi\times \P^2_\Pi$ is isomorphic to either $\P^{2|3}\times \P^2_\Pi$ or
to the quotient $\P^{2|3}\times \P^{2|3}$ by the free $\A^{0|1}$-action corresponding to the homological vector field $v_1+v_2$, where
$v_1$ and $v_2$ are the standard homological vector fields on each factor.
\end{lemma}

\begin{proof}
First, it easy to check that $\P^2_\Pi\times \P^2_\Pi$ that $H^1(\P^2_\Pi\times \P^2_\Pi,\OO^-)$ is $2$-dimensional. More precisely,
it is the dirrect sum of two copies of $H^1(\Pi^2_\Pi,\OO^-)$, which is $1$-dimensional. The generator of $H^1(\Pi^2_\Pi,\OO^-)$ corresponds to the
standard $\A^{0|1}$-torsor $\P^{2|3}\to \P^2_\Pi$, where the $\A^{0|1}$-action on $\P^{2|3}$ is given by the homological vector field $v$. 
Any multiple of the generator gives an $\A^{0|1}$-torsor corresponding to the same projection $\P^{2|3}\to \P^2$ but with the $\A^{0|1}$-action differing
by an automorphism of $\A^{0|1}$. 

It remains to observe that the Baer sum of two $\A^{0|1}$-torsors over $\P^2_\Pi\times \P^2_\Pi$ pulled back from the two factors is obtained as the quotient
of $\A^{0|1}\times \A^{0|1}$-torsor, given by their direct product $\P^{2|3}\times \P^{2|3}\to \P^2_\Pi\times \P^2_\Pi$, by the antidiagonal subgroup 
$\A^{0|1}\sub \A^{0|1}\times \A^{0|1}$. This easily leads to the claimed description.
\end{proof}

\begin{lemma}\label{Pic-product-lem}
Let $X$ be a proper superscheme over an algebraically closed field $k$, such that $H^0(X,\OO_X)=k$ and $H^1(X,\OO_X)=0$.
Then for any connected superscheme $Y$ of finite type over $k$ the natural map $\Pic(X)\times \Pic(Y)\to \Pic(X\times Y)$ is an isomorphism.
\end{lemma}

\begin{proof} Let $L$ be a line bundle on $X\times Y$. By tensoring $L$ with the pull-back of a line bundle on $X$, we can
assume that $L$ has trivial restriction to $X\times p_0$.
First, we claim that for every point $p$ in $Y$, the restriction of $L$ to $X\times p$ is trivial. Indeed, we can replace $Y$ by the corresponding reduced scheme $Y_0$.
Then we observe that any line bundle on $X\times Y_0$ is canonically the pull-back from the corresponding bosonic quotient, which is $(X/\Ga)\times Y_0$, where
$X/\Ga$ is the bosonic quotient of $X$. Thus, our claim follows from the similar result for usual schemes (see e.g. \cite[Ex.\ III.12.6]{Hart}).

It follows that $p_{2*}L$ is a line bundle on $Y$. Now we claim that the morphism of line bundles $p_2^*p_{2*}L\to L$ is an isomorphism.
Indeed, it is sujrective since this can be checked fiberwise.
\end{proof}

\begin{proof}[Proof of Proposition \ref{prod-Pi-proj-planes-prop}]
First, we claim that $Y=\P^2_\Pi\times \P^2_\Pi$ has trivial Picard group. For this, let us view $Y$ as the quotient of $X=\P^{2|3}\times \P^2_\Pi$ by the $\A^{0|1}$-action
(on the first factor).
We claim that $\Pic(X)=\Z$. Indeed, we have $H^1(\P^{2|3},\OO)=0$ and $\Pic(\P^2_\Pi)=0$ (see \cite[Thm.\ 5.2]{CNR}).
Hence, by Lemma \ref{Pic-product-lem}, one has
$$\Pic(X)=\Pic(\P^{2|3})=\Z.$$
Next, we recall that the line bundle $\OO(n)$ with $n\neq 0$ on $\P^{2|3}$ has a $v$-connection with nonzero curvature. Hence,
the same is true for the line bundle $p_1^*\OO(n)$ on $X$. By Lemma \ref{v-conn-lem}, this implies that this line bundle does not descend to $Y$.

By Proposition \ref{non-1|1-emb-prop-1}(i), it remains to check that $\A^{0|1}$-torsors over $Y$ are not projective. Since $\Pi^2_\Pi$ is not projective,
by Lemma \ref{prod-Pi-proj-planes-tors-lem}, we have to check that the quotient $\P^{2|3}\times \P^{2|3}/\A^{0|1}$, where the $\A^{0|1}$-action is given by
$v_1+v_2$, is not projective. By Lemma \ref{Pic-product-lem}, any line bundle on $\P^{2|3}\times \P^{2|3}$ is of the form $\OO(n_1)\boxtimes \OO(n_2)$. We claim that
such a line bundle descends to the quotient only if $n_1+n_2=0$. Since none of these bundles restricts to an ample line bundle on the reduced space, this would show
the required non-projectivity. Recall that we have natural $v_i$-connections $\nabla_i$ on $\OO(n_i)$, for $i=1,2$, with the curvature $-n_i$.
Hence, we get the induced $v_i$-connections $\nabla_i$ on $\OO(n_1)\boxtimes \OO(n_2)$, still with the curvature $-n_i$.
Now $\nabla_1+\nabla_2$ is a $v_1+v_2$-connection, with the curvature $-(n_1+n_2)$. Thus, by Lemma \ref{v-conn-lem}(ii), the line bundle $\OO(n_1)\boxtimes \OO(n_2)$
does not admit a flat $v_1+v_2$-connection unless $n_1+n_2=0$.
\end{proof}

\begin{cor} $G(1|1,m|m)\times G(1|1,n|n)$ is not $1|1$-embeddable for $m\ge 3$, $n\ge 3$.
\end{cor}

\begin{proof} Indeed, this follows from the closed embedding of $\P^2_\Pi$ into $G(1|1,3|3)$, and hence into $G(1|1,n|n)$ for $n\ge 3$.
\end{proof}

\end{document}